\documentclass[10pt,reqno]{amsart}

\pdftrailerid{}
\usepackage[T1]{fontenc}
\usepackage{lmodern}
\usepackage[margin=1.08in]{geometry}
\usepackage{microtype}
\usepackage{amsmath,amssymb,mathtools,mathrsfs}
\usepackage{enumitem}
\usepackage[hidelinks]{hyperref}
\hypersetup{
  pdftitle={Entire Functions Mapping Countable Dense Subsets of R
    onto Countable Dense Subsets of C},
  pdfauthor={Sina Nadi},
  pdfsubject={Entire Functions Mapping Countable Dense Sets},
  pdfcreator={},
  pdfproducer={},
  pdfkeywords={},
}

\setlist[enumerate]{label=(\roman*),leftmargin=2.15em,itemsep=0.15em,topsep=0.35em}
\setlist[itemize]{leftmargin=1.8em,itemsep=0.15em,topsep=0.35em}
\numberwithin{equation}{section}

\newtheorem{theorem}{Theorem}[section]
\newtheorem{lemma}[theorem]{Lemma}
\newtheorem{proposition}[theorem]{Proposition}
\newtheorem{corollary}[theorem]{Corollary}
\theoremstyle{remark}

\newcommand{\C}{\mathbb C}
\newcommand{\R}{\mathbb R}
\newcommand{\Q}{\mathbb Q}
\newcommand{\dist}{\operatorname{dist}}
\newcommand{\normone}[2]{\lVert #1\rVert_{C^1(#2)}}

\newcommand{\Real}[1]{\mathit{Re}\{#1\}}
\newcommand{\lemref}[1]{\hyperref[#1]{Lemma~\ref*{#1}}}
\newcommand{\thmref}[1]{\hyperref[#1]{Theorem~\ref*{#1}}}
\newcommand{\corref}[1]{\hyperref[#1]{Corollary~\ref*{#1}}}

\title[Entire functions mapping countable dense sets]
{Entire Functions Mapping Countable Dense Subsets of $\mathbb R$\\
onto Countable Dense Subsets of $\mathbb C$}
\author[Sina Nadi]{Sina Nadi}
\date{}

\begin{document}

\begin{abstract}
In 1977, Karl F. Barth posed the following problem: given countable dense
sets $A\subset\R$ and $B\subset\C$, does there exist a transcendental entire
function $f$ such that
\[
    f(A)=B
    \qquad\text{and}\qquad
    f(\R\setminus A)\subset\C\setminus B?
\]
We review results related to this question and prove that there exist
transcendental entire functions $f$ such that
$f\!\restriction_A\colon A\to B$ is bijective,
$f^{-1}(B)\cap\R=A$, and
$f'(a)\neq0$ for every $a\in A$.  In fact, the set of such functions has
the cardinality of the continuum.  At the end, we give two extensions of the
result, one for countably many pairwise disjoint pairs of dense sets and one
with $\R$ replaced by a closed unbounded subset of $\C$ of planar Lebesgue
measure zero.
\end{abstract}

\maketitle

\section{Introduction}

The arithmetic behavior of transcendental entire functions has long been studied in
complex analysis and transcendence theory.  A related problem in real and complex
analysis asks how an entire function can act on countable dense subsets of $\R$ and
$\C$.

In 1886, Strauss considered whether a transcendental analytic function could take
rational values at every rational point of its domain.  He tried to prove that this
was impossible.  Weierstrass answered Strauss in the same year by giving a
transcendental entire function that takes rational values at all rational points.  He
also suggested that a transcendental entire function could take algebraic values at
every algebraic point.  St\"ackel recorded this exchange in his 1895 paper \cite{Stackel1895}.
Mahler later discussed it in his 1976 lectures \cite[p.~53]{Mahler1976}.

In 1895, Cantor proved that every countable dense linear order without endpoints
has the same order type as the rational numbers.  It follows that any two countable
dense subsets of $\R$ are related by an order-preserving bijection
\cite{Cantor1895}.  In the same year, St\"ackel proved in \cite{Stackel1895} that if $X\subset\C$ is
countable and $Y\subseteq\C$ is dense, then there exists a transcendental entire
function $f$ such that
\[
    f(X)\subseteq Y.
\]

In 1902, St\"ackel constructed a transcendental function that is analytic in a
neighborhood of the origin and has a local inverse there.  He arranged that both
the function and its local inverse take algebraic values at every algebraic point
for which they are defined \cite{Stackel1902}.

In 1925, Franklin started from Cantor's theorem and considered countable dense
subsets $A$ and $B$ of two open real intervals $I$ and $J$.  He stated that there
exists a real-analytic function $f$ from $I$ onto $J$ such that
\[
    f'(x)>0\qquad(x\in I)
\]
and
\[
    f(A)=B.
\]
Thus $f$ is an increasing bijection from $I$ onto $J$, and its restriction to $A$
is an order-preserving bijection from $A$ onto $B$ \cite{Franklin1925}.

In 1957, Erd\H{o}s posed Problem~24 in \cite{Erdos1957}.  He first asked whether there exists an entire
function $f$, not of the form $a_0+a_1z$, such that for every real number $x$,
\[
    f(x)\in\Q
    \quad\Longleftrightarrow\quad
    x\in\Q.
\]
He then asked whether, for two denumerable dense sets $A$ and $B$, there exists an
entire function that maps $A$ onto $B$.  The second question did
not specify whether the two sets were dense in the real line or in the complex
plane.

In 1963, Neumann and Rado answered the first part of Erd\H{o}s's problem.  They
constructed a transcendental entire function that takes rational values at rational
real points and irrational values at irrational real points
\cite{NeumannRado1963}.

In 1967, Maurer answered the complex-plane interpretation of the second part.  For
any countable dense sets $A,B\subset\C$, he constructed a transcendental entire
function $f$ whose restriction to $A$ is a bijection from $A$ onto $B$
\cite{Maurer1967}.

In 1970, Barth and Schneider answered the real-line interpretation
\cite{BarthSchneider1970}.  For any countable dense sets $A,B\subset\R$, they
constructed a transcendental entire function $f$ such that
\[
    f(z)\in B
    \quad\Longleftrightarrow\quad
    z\in A.
\]
They also arranged that the restriction of $f$ to $\R$ is a strictly increasing
homeomorphism of $\R$ onto itself and that
\[
    f(A)=B.
\]

In 1972, Barth and Schneider strengthened Maurer's theorem.  For any countable
dense sets $A,B\subset\C$, they constructed an entire function $f$ such that
\[
    f(z)\in B
    \quad\Longleftrightarrow\quad
    z\in A,
\]
and
\[
    f(A)=B,
    \qquad
    f(\C\setminus A)=\C\setminus B.
\]
They noted that this result does not imply their 1970 theorem because it does not
ensure that $f$ maps the real line onto itself \cite{BarthSchneider1972}.

In 1974, Sato and Rankin gave another solution of the real-line problem.  They
proved in \cite{SatoRankin1974} that for any countable dense sets
$A,B\subset\R$, there exists a transcendental entire function $f$ whose
restriction to $\R$ is a real-valued strictly increasing surjection and satisfies
\[
    f(A)=B.
\]

In 1976, Nienhuys and Thiemann proved the following result in
\cite{NienhuysThiemann1976}.  Let $S,T\subset\R$ be countable and dense.  Let
$p$ be a positive continuous function on $[0,\infty)$ such that
\[
    \lim_{t\to\infty}t^{-n}p(t)=\infty
\]
for every positive integer $n$, and let $f_0$ be an entire function whose
restriction to $\R$ is real-valued and nondecreasing.  They proved that there
exists an entire function $f$ whose restriction to $\R$ is strictly increasing,
such that
\[
    f(S)=T
\]
and
\[
    |f(z)-f_0(z)|<p(|z|)\qquad(z\in\C).
\]

In 1976, Mahler returned to St\"ackel's 1902 result.  He asked whether there exists
a transcendental entire function with rational Taylor coefficients such that both
the image and the preimage of the field of algebraic numbers are contained in the
field of algebraic numbers \cite[p.~53]{Mahler1976}.

In 1977, Anderson, Barth, and Brannan asked whether, for countable dense sets
$A\subset\R$ and $B\subset\C$, there exists a transcendental entire function
$f$ that maps $A$ onto $B$ and maps $\R\setminus A$ into $\C\setminus B$.
They attributed the problem to Barth \cite{AndersonBarthBrannan1977}.

In 2009, Burke strengthened the real-line theorem.  Let $A,B\subset\R$ be
countable and dense.  Let $u$ be a nondecreasing $C^n$ surjection from $\R$ onto
itself, and let $\varepsilon$ be a positive continuous function on $\R$.  Burke
proved that one can choose an entire function $f$, real-valued and increasing on
$\R$, such that
\[
    f(A)=B
\]
and
\[
    \left|f^{(k)}(x)-u^{(k)}(x)\right|<\varepsilon(x)
\]
for $0\leq k\leq n$ and $x\in\R$ \cite{Burke2009JMAA}.

In 2017, Marques and Moreira answered Mahler's 1976 question.  They proved that
there are uncountably many transcendental entire functions with rational Taylor
coefficients for which the image and the preimage of the field of algebraic numbers
are contained in that field \cite{MarquesMoreira2017}.

In 2019, Gauthier revisited the real-line problem.  He pointed out that Franklin's
proof used the false assertion that a uniformly convergent sequence of analytic
functions on a real interval must have an analytic limit.  Franklin's conclusion
remains valid because it follows from Burke's 2009 theorem.  Gauthier also proved
that, for any countable dense sets $A,B\subset\R$, the function may be chosen to be
entire of finite order, to satisfy
\[
    f(\R)=\R,
    \qquad
    f'(x)>0\qquad(x\in\R),
\]
and to restrict to an order-preserving bijection from $A$ onto $B$
\cite{Gauthier2019}.

In 2019, Hayman and Lingham recorded the following problem, due to Karl F.
Barth, as Problem~2.48 in the fiftieth-anniversary edition of
\emph{Research Problems in Function Theory} \cite{HaymanLingham2019}.

\medskip
\noindent\textbf{Barth's problem.}
Let $A\subset\R$ and $B\subset\C$ be countable dense sets.  Does there exist a
transcendental entire function $f$ that maps $A$ onto $B$ and maps
$\R\setminus A$ into $\C\setminus B$?
\medskip

The two conditions in Barth's problem may be written as
\[
    f(A)=B
\]
and
\[
    f^{-1}(B)\cap\R=A.
\]

Earlier in 2019, Marques and Moreira had proved a strengthened version of their
2017 theorem.  Their main theorem concerns countable dense subsets of $\C$
satisfying certain symmetry conditions and entire functions with arithmetic
restrictions on their coefficients.  In Remark~1 of \cite{MarquesMoreira2019},
they state that when arbitrary complex coefficients are allowed, the symmetry
conditions can be removed.  More precisely, they state that for arbitrary
countable dense sets $X,Y\subset\C$, there are uncountably many transcendental
entire functions $g$
such that
\[
    g(X)=Y,
    \qquad
    g^{-1}(Y)=X,
    \qquad
    g'(\alpha)\neq0
    \quad(\alpha\in X).
\]

Remark~1 of Marques and Moreira gives part of the conclusion required in
Barth's problem.  Choose a countable set
\[
    D\subset\C\setminus\R
\]
that is dense in $\C$, and set
\[
    X=A\cup D,
    \qquad
    Y=B.
\]
Their result gives a transcendental entire function $g$ satisfying
\[
    g(A\cup D)=B
\]
and
\[
    g^{-1}(B)=A\cup D.
\]
Since $D\cap\R=\varnothing$, we obtain
\[
    g^{-1}(B)\cap\R=A.
\]
It follows that
\[
    g(\R\setminus A)\subset\C\setminus B
\]
and
\[
    g(A)\subseteq B.
\]
This does not imply that $g(A)=B$.  Some elements of $B$ may be attained only at
points of $D$.  Thus the result of Marques and Moreira gives the required real
preimage condition, but it does not give the surjectivity of
$g\!\restriction_A$ onto $B$.  The purpose of the present paper is to prove that
both requirements in Barth's problem can always be satisfied simultaneously.  We
also extend the result to countably many pairwise disjoint pairs of dense sets
and to closed unbounded subsets of $\C$ of planar Lebesgue measure zero.

\medskip

Our results give a stronger answer.

\begin{theorem}\label{thm:main}
Let $A\subset\R$ and $B\subset\C$ be countable dense sets.  There exists a
transcendental entire function $f$ such that
\begin{enumerate}
    \item $f\!\restriction_A\colon A\to B$ is a bijection;
    \item $f^{-1}(B)\cap\R=A$.
\end{enumerate}
Equivalently, every $b\in B$ has exactly one real preimage, and the collection of
those preimages is $A$.
\end{theorem}

We prove the theorem in the next section.  We also prove that the functions
may be chosen so that $f'(a)\neq0$ for every $a\in A$, and the set of such
functions has the cardinality of the continuum.  Two extensions of the theorem
are given at the end of the section.

\section{Proof of the theorem}

For every finite set $S\subset\R$, define
\[
    H_S(z):=\prod_{s\in S}(z-s)\in\C[z].
\]
We adopt the conventions $H_\varnothing\equiv1$ and $\deg 0=-\infty$, where $0$ denotes the zero polynomial and $\deg$ denotes the degree of a polynomial.  For every nonempty compact
set $K\subset\C$ and every entire function $h$, define
\[
    \normone{h}{K}
      :=\sup_{z\in K}|h(z)|+\sup_{z\in K}|h'(z)|.
\]

For $a\in\C$ and $r>0$, let
$D(a,r):=\{z\in\C:|z-a|<r\}$ and
$\overline D(a,r):=\{z\in\C:|z-a|\leq r\}$.

\begin{lemma}\label{lem:add-source}
Let $p\in\C[z]$, let $S\subset\R$ be finite, let $a\in\R\setminus S$, and
let $T\subset\C$ be finite.  Assume that $p(S)\subseteq T$.  Given a nonempty
compact set $K\subset\C$ and $\varepsilon>0$, there exist
$b\in B\setminus T$ and $q\in\C[z]$ such that
\begin{enumerate}
    \item $q(s)=p(s)$ for every $s\in S$;
    \item $q(a)=b$;
    \item $\normone{q-p}{K}<\varepsilon$.
\end{enumerate}
\end{lemma}

\begin{proof}
Since $H_S(a)\neq0$ and $\normone{H_S}{K}>0$, and since $B\setminus T$ is
dense in $\C$, we may choose $b\in B\setminus T$ so that
\[
    |b-p(a)|
      <\frac{\varepsilon |H_S(a)|}{\normone{H_S}{K}}.
\]
Set
\[
    q(z):=p(z)+\frac{b-p(a)}{H_S(a)}H_S(z).
\]
Then $q=p$ on $S$, while $q(a)=b$, and
\[
    \normone{q-p}{K}
      =\left|\frac{b-p(a)}{H_S(a)}\right|\normone{H_S}{K}
      <\varepsilon.
\]
\end{proof}

\begin{lemma}\label{lem:add-target}
Let $p\in\C[z]$, let $S\subset\R$ be finite, and let
$b\in\C\setminus p(S)$.  Given a nonempty compact set $K\subset\C$ and
$\varepsilon>0$, there exist $a\in A\setminus S$ and $q\in\C[z]$ such that
\begin{enumerate}
    \item $q(s)=p(s)$ for every $s\in S$;
    \item $q(a)=b$;
    \item $\normone{q-p}{K}<\varepsilon$.
\end{enumerate}
\end{lemma}

\begin{proof}
Set $h:=\deg H_S$ and $d:=\deg p$, and choose $M\geq0$ so that
\[
    h+M>\max\{d,0\}.
\]
For $a\in\R\setminus(S\cup\{0\})$, put
\[
    c_a:=\frac{b-p(a)}{H_S(a)a^M}.
\]
As $|a|\to\infty$ on the real axis,
\[
    H_S(a)a^M=a^{h+M}(1+o(1))
    \qquad\text{and}\qquad
    b-p(a)=O\bigl(|a|^{\max\{d,0\}}\bigr).
\]
The choice of $M$ therefore gives $c_a\to0$.

Let
\[
    R_M(z):=H_S(z)z^M,
    \qquad
    M_K:=\max\{1,\normone{R_M}{K}\}.
\]
The density of $A$ implies that $A\setminus(S\cup\{0\})$ is unbounded.  We may
therefore choose $a$ in this set, with $|a|$ sufficiently large, so that
\[
    |c_a|M_K<\varepsilon.
\]
Set
\[
    q(z):=p(z)+c_aR_M(z).
\]
Then $q=p$ on $S$, $q(a)=b$, and
\[
    \normone{q-p}{K}
      =|c_a|\normone{R_M}{K}
      \leq|c_a|M_K
      <\varepsilon.\qedhere
\]
\end{proof}

\begin{lemma}\label{lem:clearing}
Let $q\in\C[z]$, let $S\subset\R$ and $E\subset\C$ be finite, and assume that
$E\neq\varnothing$ and $q(S)\subseteq E$.  Let $I=[u,v]\subset\R$, where
$u<v$.  For every nonempty compact set $K\subset\C$ and every
$\varepsilon>0$, there exists $p\in\C[z]$ such that
\begin{enumerate}
    \item $p(s)=q(s)$ for every $s\in S$;
    \item $p^{-1}(E)\cap I=S\cap I$;
    \item $p'(s)\neq0$ for every $s\in S$;
    \item $\normone{p-q}{K}<\varepsilon$.
\end{enumerate}
\end{lemma}

\begin{proof}
Use the same polynomial $H_S$.  For $t\in\C$, define
\[
    p_t(z):=q(z)+tH_S(z).
\]
This perturbation fixes the values of $q$ at every point of $S$.

Fix $e\in E$.  If $x\in I\setminus S$, then $H_S(x)\neq0$, and the equation
$p_t(x)=e$ is equivalent to
\[
    t=\Phi_e(x),
    \qquad
    \Phi_e(x):=\frac{e-q(x)}{H_S(x)}.
\]
The set $I\setminus S$ has finitely many connected components, each of which
is a countable union of compact subintervals on which $H_S$ does not vanish.
On every such subinterval $J$, the map $\Phi_e$ is of class $C^1$.  The image
of a compact interval under a $C^1$ map into $\C$ has planar Lebesgue measure
zero.  Hence, by countable subadditivity,
\[
    m_2\bigl(\Phi_e(I\setminus S)\bigr)=0,
\]
where $m_2$ denotes planar Lebesgue measure.

The set of all parameters that create an unwanted point of $p_t^{-1}(E)$ in
$I\setminus S$ is
\[
    \mathcal B
      :=\bigcup_{e\in E}\Phi_e(I\setminus S).
\]
The set $E$ is finite, so $m_2(\mathcal B)=0$.

We also exclude those parameters $t$ for which $p_t'(s)=0$ at some $s\in S$.  For
every $s\in S$,
\[
    p_t'(s)=q'(s)+tH_S'(s).
\]
The derivative of $H_S$ at $s$ is
\[
    H_S'(s)=\prod_{u\in S\setminus\{s\}}(s-u),
\]
and this product is nonzero.  Therefore precisely one parameter can satisfy
$p_t'(s)=0$, namely
\[
    t_s:=-\frac{q'(s)}{H_S'(s)}.
\]
Define
\[
    \mathcal E:=\mathcal B\cup\{t_s:s\in S\}.
\]
The set $\mathcal E$ has planar measure zero.  In particular, it cannot contain any nonempty open disc.

Set $C_K:=\normone{H_S}{K}$.  If $S=\varnothing$, then $H_S\equiv1$.
If $S\neq\varnothing$, all zeros of $H_S$ are simple, so $H_S$ and $H_S'$ do
not vanish simultaneously.  Since $K$ is nonempty, in either case $C_K>0$.
The disc
\[
    D\left(0,\frac{\varepsilon}{C_K}\right)
\]
contains a point $t\notin\mathcal E$.  For this choice,
\[
    \normone{p_t-q}{K}=|t|C_K<\varepsilon.
\]
The condition $t\notin\mathcal B$ gives
\[
    p_t^{-1}(E)\cap(I\setminus S)=\varnothing.
\]
The identity $p_t(s)=q(s)\in E$ holds for every $s\in S$.  Hence
\[
    p_t^{-1}(E)\cap I=S\cap I.
\]
Finally, $t\neq t_s$ for every $s\in S$, so $p_t'(s)\neq0$.  Taking $p=p_t$
proves the lemma.
\end{proof}

\begin{lemma}\label{lem:stability}
Let $I=[u,v]\subset\R$, where $u<v$, let
$E\subset\C$ be finite and nonempty, and let $p\colon I\to\C$ be of class
$C^1$.  Suppose that
\[
    p^{-1}(E)=S
\]
for a finite set $S\subset I$, and suppose that $p'(s)\neq0$ for every
$s\in S$.  Then there exists $\eta>0$ with the following property: if
$g\colon I\to\C$ is of class $C^1$, if $g(s)=p(s)$ for all $s\in S$, and if
\[
    \sup_{x\in I}|g(x)-p(x)|+
    \sup_{x\in I}|g'(x)-p'(x)|<\eta,
\]
then
\[
    g^{-1}(E)=S,
\]
and $g'(s)\neq0$ for every $s\in S$.
\end{lemma}

\begin{proof}
If $S=\varnothing$, then the compact sets $p(I)$ and $E$ are disjoint.  Their
distance
\[
    d:=\dist(p(I),E)
\]
is positive.  The conclusion follows from any choice $0<\eta<d$, because
$\sup_I|g-p|<\eta$ then implies $g(I)\cap E=\varnothing$.  We therefore assume
that $S\neq\varnothing$.

Fix $s\in S$.  Choose $\lambda_s\in\C$ with $|\lambda_s|=1$ and
\[
    \lambda_sp'(s)=|p'(s)|.
\]
The function $x\mapsto\Real{\lambda_sp'(x)}$ is continuous and has the positive
value $|p'(s)|$ at $s$.  We may therefore choose $\delta_s>0$ and $c_s>0$ such
that, for
\[
    J_s:=I\cap[s-\delta_s,s+\delta_s],
\]
one has
\[
    \Real{\lambda_sp'(x)}\geq2c_s
    \qquad(x\in J_s).
\]
Because $S$ is finite, the numbers $\delta_s$ may be chosen so that the intervals
$J_s$ are pairwise disjoint.

Since $E$ is finite, choose $r_s>0$ such that
\[
    \overline D\bigl(p(s),3r_s\bigr)\cap E=\{p(s)\}.
\]
By decreasing $\delta_s$ if necessary, continuity of $p$ gives
\[
    p(J_s)\subset D\bigl(p(s),r_s\bigr).
\]

Since $S$ is finite and $I$ is nondegenerate, $I\setminus S\neq\varnothing$.
By decreasing the numbers $\delta_s$ further, we may also assume that
\[
    L:=I\setminus\bigcup_{s\in S}\bigl(I\cap(s-\delta_s,s+\delta_s)\bigr)
\]
is nonempty.  The set $L$ is compact and disjoint from $S$.  Since $p^{-1}(E)=S$, the compact sets
$p(L)$ and $E$ are disjoint.  Therefore
\[
    d:=\dist(p(L),E)>0.
\]

Choose $\eta>0$ such that
\[
    0<\eta<
    \min\left\{
        \frac d2,
        \min_{s\in S}r_s,
        \min_{s\in S}c_s,
        \frac12\min_{s\in S}|p'(s)|
    \right\}.
\]
Assume that $g$ satisfies the hypotheses stated in the lemma with this value of
$\eta$.

For $x\in J_s$, we have
\[
\begin{aligned}
    \Real{\lambda_sg'(x)}
      &=\Real{\lambda_sp'(x)}
        +\Real{\lambda_s(g'(x)-p'(x))} \\
      &\geq\Real{\lambda_sp'(x)}
        -|\lambda_s|\,|g'(x)-p'(x)| \\
      &=\Real{\lambda_sp'(x)}-|g'(x)-p'(x)| \\
      &>2c_s-\eta \\
      &>c_s>0.
\end{aligned}
\]
Thus the real-valued function
\[
    \psi_s(x):=\Real{\lambda_s(g(x)-p(s))}
\]
is strictly increasing on $J_s$.  Furthermore, for $x\in J_s$,
\[
\begin{split}
    |g(x)-p(s)|
      &\leq |g(x)-p(x)|+|p(x)-p(s)| \\
      &<\eta+r_s \\
      &<2r_s.
\end{split}
\]
Hence
\[
    g(J_s)\subset D\bigl(p(s),2r_s\bigr).
\]
By the choice of $r_s$,
\[
    D\bigl(p(s),2r_s\bigr)\cap E=\{p(s)\}.
\]
Hence $g(J_s)\cap E\subseteq\{p(s)\}$.  Since $g(s)=p(s)$, we have
$\psi_s(s)=0$.  If $g(x)=p(s)$ for some $x\in J_s$, then $\psi_s(x)=0$, and
the strict monotonicity of $\psi_s$ forces $x=s$.  It follows that
\[
    g^{-1}(E)\cap J_s=\{s\}.
\]

For $x\in L$,
\[
    \dist(g(x),E)
      \geq\dist(p(x),E)-|g(x)-p(x)|
      >d-\eta
      >\frac d2.
\]
Thus $g(L)\cap E=\varnothing$.  Combining this fact with the conclusions on the
intervals $J_s$ gives $g^{-1}(E)=S$.  Finally, for $s\in S$,
\[
    |g'(s)|
      \geq |p'(s)|-|g'(s)-p'(s)|
      >\frac{|p'(s)|}{2}>0.
\]
\end{proof}

\begin{proof}[Proof of \thmref{thm:main}]
Choose injective enumerations
\[
    A=\{\alpha_1,\alpha_2,\ldots\},
    \qquad
    B=\{\beta_1,\beta_2,\ldots\}.
\]
For $R>0$, define
\[
    D_R:=\{z\in\C:|z|\leq R\}.
\]
For $n\geq1$, define
\[
    I_n:=[-n,n].
\]
We construct polynomials $p_n\in\C[z]$, finite sets $S_n\subset A$ and
$T_n\subset B$, and positive numbers $\eta_n$.  Set
\[
    p_0\equiv0,
    \qquad
    S_0=T_0=\varnothing.
\]
At the end of stage $n$, the following conditions will hold:
\begin{enumerate}[label=(\alph*)]
    \item $S_{n-1}\subseteq S_n$ and $T_{n-1}\subseteq T_n$;
    \item $p_n\!\restriction_{S_n}\colon S_n\to T_n$ is a bijection;
    \item $\alpha_1,\ldots,\alpha_n\in S_n$ and
          $\beta_1,\ldots,\beta_n\in T_n$;
    \item $p_n(s)=p_{n-1}(s)$ for every $s\in S_{n-1}$;
    \item $p_n^{-1}(T_n)\cap I_n=S_n\cap I_n$;
    \item $p_n'(s)\neq0$ for every $s\in S_n$;
    \item the conclusion of \lemref{lem:stability} holds with
          $\eta=\eta_n$ for
          $(I_n,T_n,p_n\!\restriction_{I_n},S_n\cap I_n)$;
    \item
    \[
        \normone{p_n-p_{n-1}}{D_n}<\rho_n,
        \qquad
        \rho_n:=2^{-n}\min\{1,\eta_1,\ldots,\eta_{n-1}\},
    \]
    where the minimum is defined to be $1$ when $n=1$.
\end{enumerate}

Assume that the construction has been completed through stage $n-1$.  First we
insert $\alpha_n$ if it has not already been used.  If
$\alpha_n\notin S_{n-1}$, apply \lemref{lem:add-source} with
\[
    p=p_{n-1},\quad S=S_{n-1},\quad T=T_{n-1},\quad
    a=\alpha_n,\quad K=D_n,\quad \varepsilon=\frac{\rho_n}{3}.
\]
Let $b_n^{(1)}\in B\setminus T_{n-1}$ and $q_n^{(1)}\in\C[z]$ be the
point and the polynomial supplied by \lemref{lem:add-source}.  Define
\[
    S_n^{(1)}:=S_{n-1}\cup\{\alpha_n\},
    \qquad
    T_n^{(1)}:=T_{n-1}\cup\{b_n^{(1)}\}.
\]
If $\alpha_n\in S_{n-1}$, define instead
\[
    q_n^{(1)}:=p_{n-1},
    \qquad
    S_n^{(1)}:=S_{n-1},
    \qquad
    T_n^{(1)}:=T_{n-1}.
\]
In either case, $q_n^{(1)}\!\restriction_{S_n^{(1)}}$ is a bijection from
$S_n^{(1)}$ onto $T_n^{(1)}$.  If $\alpha_n\notin S_{n-1}$, then
\[
    \normone{q_n^{(1)}-p_{n-1}}{D_n}<\frac{\rho_n}{3}.
\]
If $\alpha_n\in S_{n-1}$, then $q_n^{(1)}=p_{n-1}$, and the norm on the
left is zero.

We next insert $\beta_n$ if it has not already been used.  If
$\beta_n\notin T_n^{(1)}$, then
\[
    \beta_n\notin q_n^{(1)}(S_n^{(1)}),
\]
because $q_n^{(1)}(S_n^{(1)})=T_n^{(1)}$.  Apply
\lemref{lem:add-target} with
\[
    p=q_n^{(1)},\quad S=S_n^{(1)},\quad b=\beta_n,
    \quad K=D_n,\quad \varepsilon=\frac{\rho_n}{3}.
\]
Let $a_n^{(2)}\in A\setminus S_n^{(1)}$ and $q_n^{(2)}\in\C[z]$ be the
point and the polynomial supplied by \lemref{lem:add-target}.  Define
\[
    S_n^{(2)}:=S_n^{(1)}\cup\{a_n^{(2)}\},
    \qquad
    T_n^{(2)}:=T_n^{(1)}\cup\{\beta_n\}.
\]
If $\beta_n\in T_n^{(1)}$, define
\[
    q_n^{(2)}:=q_n^{(1)},
    \qquad
    S_n^{(2)}:=S_n^{(1)},
    \qquad
    T_n^{(2)}:=T_n^{(1)}.
\]
Again, $q_n^{(2)}\!\restriction_{S_n^{(2)}}$ is a bijection from
$S_n^{(2)}$ onto $T_n^{(2)}$.  If $\beta_n\notin T_n^{(1)}$, then
\[
    \normone{q_n^{(2)}-q_n^{(1)}}{D_n}<\frac{\rho_n}{3}.
\]
If $\beta_n\in T_n^{(1)}$, then $q_n^{(2)}=q_n^{(1)}$, and the norm on the
left is zero.

Set
\[
    S_n:=S_n^{(2)},
    \qquad
    T_n:=T_n^{(2)}.
\]
Apply \lemref{lem:clearing} with
\[
    q=q_n^{(2)},\quad S=S_n,\quad E=T_n,
    \quad I=I_n,\quad K=D_n,
    \quad \varepsilon=\frac{\rho_n}{3}.
\]
The set $T_n$ is nonempty because $n\geq1$ and the first two steps ensure that
$\alpha_1\in S_n$ and $q_n^{(2)}(\alpha_1)\in T_n$.  Let
$p_n\in\C[z]$ be the polynomial supplied by \lemref{lem:clearing}.  This
polynomial fixes the values at all points of $S_n$, satisfies
\[
    p_n^{-1}(T_n)\cap I_n=S_n\cap I_n,
\]
and has $p_n'(s)\neq0$ for every $s\in S_n$.  The triangle inequality gives
\[
\begin{split}
    \normone{p_n-p_{n-1}}{D_n}
      &\leq \normone{p_n-q_n^{(2)}}{D_n}
           +\normone{q_n^{(2)}-q_n^{(1)}}{D_n}
           +\normone{q_n^{(1)}-p_{n-1}}{D_n} \\
      &<\rho_n.
\end{split}
\]
Finally, by \lemref{lem:stability}, choose $\eta_n>0$ so that its
conclusion holds for
$(I_n,T_n,p_n\!\restriction_{I_n},S_n\cap I_n)$.  This completes stage $n$.

Fix $R>0$ and choose an integer $N_R\geq R$.  For every $n\geq N_R$,
$D_R\subseteq D_n$.  Hence
\[
    \normone{p_n-p_{n-1}}{D_R}
      \leq\normone{p_n-p_{n-1}}{D_n}
      <\rho_n
      \leq2^{-n}.
\]
The series of successive differences therefore converges in the $C^1$ norm on
$D_R$.  Since $R$ was arbitrary, the sequence $(p_n)$ converges locally
uniformly on $\C$ to an entire function $f$.  The derivatives $(p_n')$ converge
locally uniformly to $f'$.

We now verify the bijection on $A$.  Let $a\in A$.  There exists $N$ with
$a\in S_N$, because the enumeration condition places $\alpha_n$ in $S_n$ by
stage $n$.  Every later polynomial fixes the values on $S_N$.  Therefore
\[
    f(a)=p_N(a)\in T_N\subset B.
\]
If $a,a'\in A$ are distinct, choose $N$ so large that both points belong to
$S_N$.  The map $p_N\!\restriction_{S_N}$ is injective, and hence
$f(a)\neq f(a')$.  Thus $f\!\restriction_A$ is injective.  Conversely, let
$b\in B$.  Choose $N$ with $b\in T_N$.  The bijection
$p_N\!\restriction_{S_N}\colon S_N\to T_N$ supplies $a\in S_N$ such that
$p_N(a)=b$.  All later polynomials fix this equality, so $f(a)=b$.  Hence
$f\!\restriction_A\colon A\to B$ is bijective.

It remains to determine the real preimages of $B$.  Fix $N\geq1$.  For every
$m>N$, the definition of $\rho_m$ gives
\[
    \rho_m\leq2^{-m}\eta_N.
\]
Since $I_N\subset D_m$, we have
\[
    \normone{p_m-p_{m-1}}{I_N}
       \leq\normone{p_m-p_{m-1}}{D_m}
       <2^{-m}\eta_N.
\]
Summing the tail gives
\[
\begin{split}
    \normone{f-p_N}{I_N}
      &\leq\sum_{m=N+1}^{\infty}
             \normone{p_m-p_{m-1}}{I_N} \\
      &<\eta_N\sum_{m=N+1}^{\infty}2^{-m} \\
      &=2^{-N}\eta_N \\
      &<\eta_N.
\end{split}
\]
Every polynomial $p_m$ with $m\geq N$ agrees with $p_N$ on $S_N$.  Passing to
the limit gives $f(s)=p_N(s)$ for $s\in S_N$.  Apply \lemref{lem:stability} with
\[
    I=I_N,\qquad E=T_N,\qquad
    p=p_N\!\restriction_{I_N},\qquad
    S=S_N\cap I_N,\qquad
    g=f\!\restriction_{I_N}.
\]
The preceding estimates verify its hypotheses, and hence
\[
    f^{-1}(T_N)\cap I_N=S_N\cap I_N.
\]

Let $x\in\R$ and suppose that $f(x)\in B$.  The sets $T_N$ increase and their
union is $B$, so there exists $N_1$ with $f(x)\in T_{N_1}$.  Choose
$N\geq\max\{N_1,|x|\}$.  Then $x\in I_N$ and $f(x)\in T_N$.  The preceding
identity gives $x\in S_N\subset A$.  We have proved
\[
    f^{-1}(B)\cap\R\subseteq A.
\]
The reverse inclusion follows from $f(A)=B$.  Therefore
\[
    f^{-1}(B)\cap\R=A.
\]

The function $f$ is nonconstant because $f(A)=B$ and $B$ is dense in $\C$.
Suppose, for a contradiction, that $f$ is a polynomial.  Then $|f(x)|\to\infty$
as $|x|\to\infty$.  We first show that $f(\R)$ is closed.  Let
$f(x_k)\to w\in\C$.  If $(x_k)$ were unbounded, then, after passing to a
subsequence, we would have $|x_k|\to\infty$, which would imply
$|f(x_k)|\to\infty$.  This contradicts the convergence of $f(x_k)$.  Thus
$(x_k)$ is bounded.  After passing to a further subsequence, we may assume that
$x_k\to x\in\R$, and continuity gives $w=f(x)$.  Hence $f(\R)$ is closed.
Since $B\subset f(\R)$ and $B$ is dense in $\C$, it would follow that
$f(\R)=\C$.

For each positive integer $m$, the set $f([-m,m])$ is the image of a compact
interval under a $C^1$ map into $\C$, and hence has planar Lebesgue measure zero.
Consequently,
\[
    f(\R)=\bigcup_{m=1}^{\infty}f([-m,m])
\]
has planar measure zero.  It cannot equal $\C$.  This contradiction proves that
$f$ is transcendental.  Finally, $B=f(A)\subset f(\R)$, and the density of
$B$ gives
\[
    \overline{f(\R)}=\C.
\]
\end{proof}

\begin{corollary}\label{cor:noncritical}
The function in \thmref{thm:main} may be chosen so that
\[
    f'(a)\neq0\qquad(a\in A).
\]
\end{corollary}

\begin{proof}
Let $f$ be the function constructed in the proof of \thmref{thm:main}, and fix
$a\in A$.  Choose $N$ so large that $a\in S_N\cap I_N$.  The tail estimate in
the proof gives
\[
    \normone{f-p_N}{I_N}<\eta_N,
\]
and $f=p_N$ on $S_N$.  The derivative conclusion of
\lemref{lem:stability}, applied on $I_N$, therefore yields $f'(a)\neq0$.
\end{proof}

\begin{corollary}\label{cor:continuum}
The set of functions satisfying the conclusion of \thmref{thm:main} and the
condition $f'(a)\neq0$ for every $a\in A$ has the cardinality of the
continuum.
\end{corollary}

\begin{proof}
Fix $z_0\in\C\setminus\R$.  For each $w\in\C$, start the construction with
$p_0\equiv w$ and use
\[
    \widetilde H_S(z):=(z-z_0)H_S(z)
\]
in place of $H_S$ in the perturbations.  Since $x-z_0\neq0$ for $x\in\R$,
all denominators and derivative factors used in the proofs remain nonzero,
while every perturbation fixes the value at $z_0$.  The resulting function
$f_w$ satisfies $f_w(z_0)=w$, so distinct values of $w$ give distinct
functions.  Notice that the set of all entire functions has cardinality at
most that of the continuum, since each is determined by its Taylor
coefficients.
\end{proof}

\begin{proposition}\label{prop:several-sets}
Let $(A_j)_{j\geq1}$ be pairwise disjoint countable dense subsets of $\R$, and
let $(B_j)_{j\geq1}$ be pairwise disjoint countable dense subsets of $\C$.
There exists a transcendental entire function $f$ such that, for every
$j\geq1$,
\[
    f\!\restriction_{A_j}\colon A_j\to B_j \text{ is bijective},
    \qquad
    f^{-1}(B_j)\cap\R=A_j.
\]
The set of such functions has the cardinality of the continuum.
\end{proposition}

\begin{proof}
Choose injective enumerations
\[
    A_j=\{\alpha_{j,1},\alpha_{j,2},\ldots\},
    \qquad
    B_j=\{\beta_{j,1},\beta_{j,2},\ldots\},
\]
and enumerate the pairs $(j,k)$.  In the construction used in the proof of
\thmref{thm:main}, keep finite sets $S_{j,n}\subset A_j$ and
$T_{j,n}\subset B_j$, only finitely many of which are nonempty, and require
the polynomial at stage $n$ to map each $S_{j,n}$ bijectively onto
$T_{j,n}$.

When the pair $(j,k)$ is considered, apply \lemref{lem:add-source} with $B_j$ in place of
$B$ to include $\alpha_{j,k}$.  If $\beta_{j,k}$ has not yet been included in
$T_{j,n}$, then it belongs to none of the sets $T_{\ell,n}$, since
$T_{\ell,n}\subset B_\ell$ and the sets $B_\ell$ are pairwise disjoint.  It
is therefore not in the image of $\bigcup_{\ell\geq1} S_{\ell,n}$, and \lemref{lem:add-target}
applies with $A_j$ in place of $A$.  \hyperref[lem:clearing]{Lemmas~\ref*{lem:clearing}} and \hyperref[lem:stability]{\ref*{lem:stability}} are applied to
\[
    S_n=\bigcup_{j\geq1}S_{j,n},
    \qquad
    T_n=\bigcup_{j\geq1}T_{j,n}.
\]
The estimates and the convergence argument are unchanged.  Since every pair
$(j,k)$ is eventually considered, the restriction of the limit function to
$A_j$ is a bijection onto $B_j$ for every $j$.

Suppose that $x\in\R$ and $f(x)\in B_j$.  Choose $N$ so that $x\in I_N$ and
$f(x)\in T_{j,N}$.  As in the proof of \thmref{thm:main},
\[
    f^{-1}(T_N)\cap I_N=S_N\cap I_N.
\]
Thus $x\in S_{\ell,N}$ for some $\ell$.  Every later polynomial agrees with
$p_N$ on $S_N$, so
\[
    f(x)=p_N(x)\in T_{\ell,N}.
\]
Since also $f(x)\in T_{j,N}$, pairwise disjointness gives $\ell=j$.  Hence
$f^{-1}(B_j)\cap\R\subset A_j$, and the reverse inclusion follows from
$f(A_j)=B_j$.

Since $B_1\subset f(\R)$, the proof of transcendence in
\thmref{thm:main} applies and also gives $\overline{f(\R)}=\C$.  The
modification in the proof of \corref{cor:continuum} applies without
change and shows that the set of such functions has the cardinality of the
continuum.
\end{proof}

\begin{proposition}\label{prop:carrier}
Let $M\subset\C$ be closed, unbounded, and of planar Lebesgue measure zero.
Let $A\subseteq M$ be countable and dense in $M$, and let $B\subset\C$ be
countable and dense.  There exists a transcendental entire function $f$ such
that
\[
    f\!\restriction_A\colon A\to B \text{ is bijective},
    \qquad
    f^{-1}(B)\cap M=A.
\]
The set of such functions has the cardinality of the continuum.
\end{proposition}

\begin{proof}
Set $X_n=M\cap D_n$ and use $X_n$ in place of $I_n$ in the proof of
\thmref{thm:main}.  The set $A$ is unbounded.  The proof of \lemref{lem:add-source} applies without change
for finite $S\subset M$.  For \lemref{lem:add-target} choose $L\geq0$ so that
$|S|+L>\max\{\deg p,0\}$.  Then
\[
    \frac{b-p(a)}{H_S(a)a^L}\longrightarrow0
    \qquad (|a|\to\infty,\ a\in A),
\]
so \lemref{lem:add-target} applies as well.

For \lemref{lem:clearing}, use the same perturbation $p_t=q+tH_S$.  For $e\in E$ and $m\geq1$, the set
\[
    Y_{n,m}=\{z\in X_n:|H_S(z)|\geq1/m\}
\]
is compact, and the function $z\mapsto(e-q(z))/H_S(z)$ is Lipschitz on it.
The image of $Y_{n,m}$ therefore has planar Lebesgue measure zero.  Since
\[
    X_n\setminus S=\bigcup_{m\geq1}Y_{n,m},
\]
the set of values of $t$ for which $p_t(z)=e$ for some
$z\in X_n\setminus S$ and some $e\in E$ has measure zero.  After excluding
the finitely many values for which $p_t'(s)=0$ at a point $s\in S$, one can
choose $t$ arbitrarily small.  This proves the analogue of \lemref{lem:clearing} on
$X_n$.

The conclusion of \lemref{lem:stability} also holds on $X_n$.  Suppose that
\[
    p^{-1}(E)\cap X_n=S\cap X_n,
    \qquad
    p'(s)\neq0\quad(s\in S\cap X_n).
\]
If $X_n=\varnothing$, there is nothing to prove.  If
$S\cap X_n=\varnothing$, use the positive distance between the compact set
$p(X_n)$ and $E$.  Otherwise, for each $s\in S\cap X_n$, choose
$\lambda_s\in\C$ with $|\lambda_s|=1$ and
$\lambda_sp'(s)=|p'(s)|$.  Choose pairwise disjoint closed discs $\Delta_s$
about these points and numbers $r_s>0$ so that
\[
    \Real{\lambda_sp'(z)}>0\quad(z\in\Delta_s\cap D_n),
    \qquad
    D\bigl(p(s),2r_s\bigr)\cap E=\{p(s)\},
\]
and
\[
    p(X_n\cap\Delta_s)\subset D\bigl(p(s),r_s\bigr).
\]
Let
\[
    L=X_n\setminus
    \bigcup_{s\in S\cap X_n}\operatorname{int}\Delta_s,
\]
where $\operatorname{int}\Delta_s$ denotes the interior of $\Delta_s$.
If $L\neq\varnothing$, then $p(L)\cap E=\varnothing$, so $p(L)$ has
positive distance from $E$.  Let $g$ be holomorphic on a neighborhood of
$D_n$, and suppose that $g$ agrees with $p$ on $S\cap X_n$ and is
sufficiently close to $p$ in the $C^1$ norm on $D_n$.  For
$z\in X_n\cap\Delta_s$ with $z\neq s$, we have
\[
    g(z)-g(s)=(z-s)\int_0^1g'\bigl(s+t(z-s)\bigr)\,dt\neq0,
\]
because the integral has positive real part after multiplication by
$\lambda_s$.  Uniform closeness and the choice of $r_s$ exclude the other
points of $E$ in these discs.  If $L\neq\varnothing$, uniform closeness
excludes points of $E$ on $L$.  Thus
\[
    g^{-1}(E)\cap X_n=S\cap X_n,
\]
and the same estimate gives $g'(s)\neq0$ for $s\in S\cap X_n$.

The remainder of the proof and all estimates are unchanged.  They give an
entire function $f$ for which $f\!\restriction_A$ is a bijection onto $B$ and
$f^{-1}(B)\cap M=A$.  The function $f$ is nonconstant because $f(A)=B$ and
$B$ is dense in $\C$.  If $f$ were a polynomial, then $f(M)$ would be
closed, since $|f(z)|\to\infty$ as $|z|\to\infty$.  As $B\subset f(M)$, this
would give $f(M)=\C$.  On the other hand,
\[
    M=\bigcup_{n\geq1}X_n,
\]
and $f$ is Lipschitz on each $D_n$.  Thus $f(M)$ has planar Lebesgue measure
zero, a contradiction.  Since $B=f(A)\subset f(M)$, we also have
$\overline{f(M)}=\C$.

To see that the set of such functions has the cardinality of the continuum,
choose $z_0\in\C\setminus M$ and make the modification used in the proof of
\corref{cor:continuum}.  The factor
$z-z_0$ does not vanish on $M$, so all the preceding arguments remain valid,
and the resulting functions satisfy $f_w(z_0)=w$.
\end{proof}

\section*{Acknowledgments}
We thank Alexandre Eremenko for a careful reading of the manuscript and for
helpful comments and suggestions.

\end{document}